\documentclass[11pt,a4paper,twoside]{article}
\usepackage{amssymb,amsthm}
\usepackage{amsmath}
\usepackage{hyperref}
\usepackage{setspace}
\usepackage[a4paper,top=33mm, bottom=27mm, left=20mm, right=20mm]{geometry}

\allowdisplaybreaks[4]

\allowdisplaybreaks[4]

\title{Decomposition of bounded degree graphs into $C_4$-free subgraphs\footnote{\noindent This work was begun in April 2014 during the Workshop on Structural Graph Theory at McGill's Bellairs Institute. We warmly thank the organisers for the collaborative opportunity.}}

\author{Ross J. Kang\footnote{\noindent Applied Stochastics, IMAPP, Radboud University Nijmegen, Netherlands. \href{mailto:ross.kang@gmail.com}{\tt ross.kang@gmail.com}. This author is supported by a NWO Veni grant.}\and Guillem Perarnau\footnote{\noindent School of Computer Science, McGill University, Montreal, Canada. \href{mailto:guillem.perarnaullobet@mcgill.ca}{\tt guillem.perarnaullobet@mcgill.ca}.}}
\date{\today}

\theoremstyle{plain}
\newtheorem{theorem}{Theorem}
\newtheorem{lemma}[theorem]{Lemma}
\newtheorem{corollary}[theorem]{Corollary}

\theoremstyle{definition}
\newtheorem{indexedclaim}[theorem]{Claim}

\renewenvironment{proof}[1][Proof]{\begin{trivlist}
\item[\hskip\labelsep {\textit{#1}.}]}{\qed\end{trivlist}}

\newcommand{\G}{\mathcal{G}}
\newcommand{\E}{\mathbb{E}}

\renewcommand{\a}{\alpha}
\newcommand{\ex}{\text{ex}}

\begin{document}

\pagenumbering{arabic}

\setcounter{section}{0}

\maketitle

\begin{abstract}
We prove that every graph with maximum degree $\Delta$ admits a partition of its edges into $O(\sqrt{\Delta})$ parts (as $\Delta\to\infty$) none of which contains $C_4$ as a subgraph. This bound is sharp up to a constant factor. Our proof uses an iterated random colouring procedure.
\end{abstract}

\begin{quote}
\scriptsize Keywords: edge-colouring, graph decomposition, graphs of bounded maximum degree, degree Ramsey numbers, bipartite Tur\'an numbers.
MSC: 05C15, 05D40, 05C35, 05C55.
\end{quote}

\section{Introduction}

In this paper we consider the following question.
\begin{quote}
\em
Given a graph $G = (V,E)$ with maximum degree $\Delta$,
into how few parts can we partition $E$ so that no part has a $C_4$ subgraph?
\end{quote}
More generally, for any graph $H$ with at least two edges, given $G = (V,E)$ and a map $f:E\to [m]$ for some positive integer $m$, we call $f$ an {\em $H$-free (edge-)colouring of $G$ with $m$ colours} if there is no $i\in [m]$ such that the graph $(V,f^{-1}(i))$ contains $H$ as a subgraph.
(Note that this is not necessarily a proper colouring unless $H$ is a two-edge path.)
Let $\phi_H(G)$ be the least $m$ such that $G$ admits an $H$-free colouring with $m$ colours.

Using this notation, the above asks specifically about $\phi_{C_4}$, and in answer we show the following.
\begin{theorem}\label{thm:main}
    For every graph $G$ with maximum degree $\Delta$,
    $\phi_{C_4}(G)= O(\sqrt{\Delta})$ as $\Delta\to\infty$.
\end{theorem}
In words, every graph with maximum degree $\Delta$ admits a partition of its edges (also called a decomposition) into $O(\sqrt{\Delta})$ $C_4$-free subgraphs.

Let $K_n$ be the complete graph on $n$ vertices. By an upper bound on the size of every colour class in an $H$-free colouring of $K_{\Delta+1}$, we have that
\begin{align}\label{eq:lower}
\phi_{H}(K_{\Delta+1})\geq\frac{\binom{\Delta+1}{2}}{\ex(\Delta+1,H)}\;,
\end{align}
where $\ex(n,H)$ as usual denotes the maximum number of edges in an $H$-free graph on $n$ vertices.
Then it follows from an old result of Erd\H{o}s~\cite{Erd38} on the extremal number of $C_4$ (see~\cite{Pik12} for context and more detailed results) that $\phi_{C_4}(K_{\Delta+1})=\Omega(\sqrt{\Delta})$. This not only shows Theorem~\ref{thm:main} to be best possible up to a constant factor, but also foreshadows a central role of the complete graph.

For broader context, Theorem~\ref{thm:main} may be understood in terms of the degree Ramsey numbers as first considered in the 1970s by Burr, Erd\H{o}s and Lov\'asz~\cite{BEL76} --- they studied these numbers for complete graphs and stars. The more general setting for other graphs was recently revisited in~\cite{KMW12}. The question we posed at the beginning is equivalent to finding the multicolour degree Ramsey number of $C_4$. In~\cite{JMW13} it was shown that $\phi_{C_4}(G)=O(\Delta^{9/14})$ for graphs of maximum degree $\Delta$, and the authors asked for the right order of growth. Theorem~\ref{thm:main} settles this.

We prove Theorem~\ref{thm:main} in Section~\ref{sec:main} by using the probabilistic method. In particular, we use an iterated random colouring procedure. At each step of the procedure we identify a collection of large $C_4$-free colour classes, the removal of which significantly reduces the maximum degree of the graph (see Corollary~\ref{cor:decomp}).
In the proof, we deliberately make little effort to optimise constants, but we note here that it is possible to obtain a factor less than $45$ in Theorem~\ref{thm:main} by  being more careful at a few points.

Recently, together with Bruce Reed~\cite{PeRe14+}, the second author proved that every $\Delta$-regular graph $G$ contains a spanning $C_4$-free graph with minimum degree $\Omega(\sqrt{\Delta})$.
This result has some similarity to our Corollary~\ref{cor:decomp}, where instead of looking at the minimum degree of the resulting subgraph, they look at the minimum degree of a given colour class. In a way that is analogous to their work, we essentially reduce our considered problem to the determination of $\ex(\Delta+1,C_4)$. (For us, this is reminiscent of the relationship between independence number and chromatic number found in other extremal colouring problems.)

More generally, we ask the following.
\begin{quote}
\em
For any graph $H$ with at least two edges, is it true that $\phi_H(G)=O(\phi_H(K_{\Delta+1}))$ for every graph $G$ with maximum degree $\Delta$?
\end{quote}
Otherwise stated, we ask if the complete graph on $\Delta+1$ vertices is essentially the hardest graph to $H$-free colour among all the graphs with maximum degree $\Delta$.

Trivially, this holds for $H$ a two-edge path.
Theorem~\ref{thm:main} shows this to be true for $H=C_4$.
Using the methods in the proof of Theorem~\ref{thm:main}, it is possible to confirm this for other bipartite graphs $H$ such as cycles of order twice a prime, or complete bipartite graphs. Moreover, for every $g\geq 4$, we can also edge-colour graphs of maximum degree $\Delta$, each colour class having girth at least $g$, with an asymptotically tight number of colours. We encourage the reader to consult~\cite{PeRe14+} to see a concrete discussion of how Theorem~\ref{thm:v4} can be used to upper bound $\phi_H(G)$ for other bipartite graphs $H$.

Another problem strongly related to our result (via the above displayed question) is to determine $\phi_H(K_n)$. Inequality~\eqref{eq:lower} provides a lower bound on $\phi_H(K_n)$ in terms of $\ex(n,H)$. This prompts us to ask 
for which graphs $H$ we have $\phi_H(K_{n})=O(n^2/\ex(n,H))$ (as $n\to\infty$).

This last statement does not hold if $H$ is not bipartite. On the one hand, Tur\'an's theorem implies that $\ex(n,H)=\Omega(n^2)$. On the other hand, it can be shown in this case that $\phi_H(K_{n})=\Omega(\log\log{n})$. First observe that $\phi_H(K_{n})\geq \phi_{H'}(K_{n})$ for any $H\subseteq H'$. Write $|V(H)|=k$ for some fixed $k\ge3$. The Erd\H{o}s--Szekeres bound on two-colour Ramsey numbers gives that $R(k,\ell)\leq \binom{k+\ell-2}{k-1} = O(\ell^{k-1})$, so every $K_k$-free graph of order $n$ has an independent set of size $\Omega(n^{1/(k-1)})$. Let $m=\phi_{K_k}(K_{n})$ and let $G_1,\dots,G_m$ denote the colour classes of a $K_k$-free colouring of $K_{n}$ with $m$ colours. Beginning with $V_0=V$, define $V_{i}$ to be a maximum independent set of $G_{i}[V_{i-1}]$ for every $0<i\le m$. Then $|V_i|=\Omega(n^{(k-1)^{-i}})$, which implies $m=\Omega(\log\log{n})$, as claimed.

Nevertheless, $\phi_H(K_{n})=O(n^2/\ex(n,H))$ for some bipartite graphs $H$ such as $C_4$~\cite{ChGr75,Irv74}, $C_6$ and $C_{10}$~\cite{LL09}.

 Bounding or determining the Tur\'an number of bipartite graphs is a central problem in extremal graph theory (see again~\cite{Pik12} or, more generally,~\cite{FuSi13}), so determining for bipartite $H$ the right order of $\phi_H(G)$ in terms of $\Delta(G)$ might be difficult in general.

\section{Some probabilistic tools}

For our proof we need the following lemmas, the uses of which are covered extensively in~\cite{MoRe02}.

\begin{lemma}[Simple Concentration Bound]
Let $X$ be a random variable determined by $n$ trials $T_1,\dots,T_n$ such that for each $i$, and any two possible sequences of outcomes $t_1,\dots,t_i,\dots,t_n$ and  $t_1,\dots,t'_i,\dots,t_n$,
$$
|X(t_1,\dots,t_i,\dots,t_n)-X(t_1,\dots,t'_i,\dots,t_n)|\leq c\;.
$$
Then
\[
\Pr(|X-\E(X)|>t)\leq 2 e^{-t^2/(2c^2n)}\;.
\]
\end{lemma}

\begin{lemma}[Lov\'asz Local Lemma]\label{lem:LLL}
Consider a set $\mathcal{E}$ of events such that for each $E\in\mathcal{E}$
\begin{itemize}
\item $\Pr(E)\leq p< 1$, and
\item $E$ is mutually independent from the set of all but at most $D$ of other events.
\end{itemize}
If $4pD\leq 1$, then with positive probability none of the events in $\mathcal{E}$ occur.
\end{lemma}

\section{Proof of Theorem~\ref{thm:main}}\label{sec:main}

Before proceeding with the main proof, let us first consider the complete graph $K_{\Delta+1}$.
It was shown in the 1970s independently by Chung and Graham~\cite{ChGr75} and by Irving~\cite{Irv74} that, if $\Delta=p^2+p+1$ for some prime power $p$, then $\phi_{C_4}(K_{\Delta+1})\le p+1$.

By the density of the primes, it follows easily that
\begin{align}\label{eq:clique_decomp}
\phi_{C_4}(K_{\Delta+1})\le\lceil2\sqrt{\Delta}\rceil\;,
\end{align}
for all large enough $\Delta$.
We later use this in the proof of Theorem~\ref{thm:main}.

Given a graph $G = (V,E)$, we say that a map $f:V\to [m]$ is {\em $1$-frugal} if it holds for all $i\in [m]$ and $v\in V$ that $|f^{-1}(i)\cap N(v)|\le 1$.
We may alternatively view a $1$-frugal map as a vertex colouring such that every neighbourhood is {\em rainbow}.
The engine in our proof of Theorem~\ref{thm:main} is the following result.
\begin{theorem}\label{thm:v4}
    Let $G=(V,E)$ be a graph with maximum degree $\Delta$ and minimum degree $\delta\geq\log^2{\Delta}$ with $\Delta$ sufficiently large. For every $\a > 16$, there exist $\beta=\beta(\a) > 0$, a spanning subgraph $H$ and a (vertex) $(2\lceil\a\Delta\rceil)$-colouring $\chi$ such that
\begin{itemize}
\item $d_H(v)\geq\beta d_G(v)$ for every $v\in V$ and
\item $\chi$ is $1$-frugal and proper in $H$.
\end{itemize}
\end{theorem}
\begin{proof}
First observe that there exists a spanning bipartite subgraph $H_0$ such that $d_{H_0}(v)\geq d_G(v)/2$ for every vertex $v\in V$. (Consider $H_0$ to be a subgraph induced by a maximum edge-cut. This subgraph is clearly bipartite, so let $V=A\cup B$ denote its bipartition. Suppose that $d_{H_0}(v)< d_G(v)/2$ for some $v\in V$. We can assume that $v\in A$. Then the number of edges between $A\setminus\{v\}$ and $B\cup\{v\}$ is strictly larger than the number of edges between $A$ and $B$, contradicting the maximum edge-cut assumption.)
While colouring $V$, we also construct $H$ as a subgraph of $H_0$, by sequentially removing edges.
The colouring has two consecutive rounds, the first of which colours the vertices of $A$, the second colours $B$.

We begin by describing the first round colouring $A$; this itself has two phases, a probabilistic one followed by a deterministic one.

\begin{itemize}
\item Phase I. Colour each vertex $a\in A$ with a colour $\chi_0(a)$ chosen uniformly at random from $[\lceil\a\Delta\rceil]$. From $\chi_0$ we obtain a partial colouring $\chi_1$ of $A$ as follows. We uncolour a vertex $a\in A$ if
\begin{align}\label{cond:uncolour}
|\{b\in N_{H_0}(a):\;\exists a'\in N_{H_0}(b)\setminus\{a\},\chi_0(a')=\chi_0(a)\}|\geq\frac{d_{H_0}(a)}{\sqrt{\a}};
\end{align}
that is, if $a$ certifies that too many of its neighbours have another neighbour in $A$ with colour $\chi_0(a)$. Otherwise, let $\chi_1(a)=\chi_0(a)$ and remove all edges from $a$ to $b\in N_{H_0}(a)$ where $b$ is incident to $a'$ with $a\neq a'$ and $\chi_0(a')=\chi_0(a)$.  Let $H_1$ be the subgraph obtained after removing all these edges. We have ensured that, for any $\chi_1$-coloured $a\in A$ and any $b\in N_{H_1}(a)$, $a$ is the only neighbour of $b$ coloured $\chi_1(a)$.

We stress that condition~\eqref{cond:uncolour} is always checked on the initial colouring $\chi_0$ and that all the vertices that are uncoloured lose their colour simultaneously.

\item Phase II. Order the uncoloured vertices $a_1,\dots, a_{s-1}$. For $i=1,2,\dots,s-1$, let $c\in [\lceil \a \Delta\rceil]$ be the colour minimising
\[
|\{b\in N_{H_i}(a_i):\;\exists a'\in N_{H_i}(b)\setminus\{a_i\},\chi_i(a')=c\}|\;.
\]
 Delete from $H_i$ all edges $a_i b$ such that there exists $a'\in N_{H_i}(b)\setminus\{a_i\}$ with $\chi_i(a')=c$ and call the resulting subgraph $H_{i+1}$.
Let $\chi_{i+1}$ be the partial colouring obtained from $\chi_i$ by also assigning $a_i$ the colour $c$.
\end{itemize}

First we show that $d_{H_s}(a)$ is large for every $a\in A$.
\begin{indexedclaim}\label{clm:deg_in_A}
For every $a\in A$
\[
d_{H_s}(a)\geq\left(1-\frac{1}{\sqrt{\a}}\right)d_{H_0}(a)\;.
\]
\end{indexedclaim}
\begin{proof}
Note that we only delete edges incident to $a$ at a step in the procedure when $a$ retains its colour.
If $a\in A$ retained its colour in the probabilistic phase, we can conclude $d_{H_s}(a)= d_{H_1}(a)\geq (1-1/\sqrt{\a})d_{H_0}(a)$, since by~\eqref{cond:uncolour}, conditioned on retaining the colour $\chi_0(a)$, we delete at most $d_{H_0}(a)/\sqrt{\a}$ edges incident to $a$. Otherwise, $a=a_i$ for some $i\in [s-1]$, coloured in the deterministic phase, and since there are at most $d_{H_0}(a_i)\Delta$ edges incident to $N_{H_i}(a_i)$, there exists a colour $c\in [\lceil\a\Delta\rceil]$ such that
\[
|\{b\in N_{H_i}(a_i):\;\exists a'\in N_{H_i}(b)\setminus\{a_i\},\chi_0(a')=c\}|\leq\frac{d_{H_0}(a_i)\Delta}{\lceil\a\Delta\rceil}\leq \frac{d_{H_0}(a_i)}{\a}\;.
\]
Thus $d_{H_s}(a_i)=d_{H_{i+1}}(a_i)\geq (1-1/\a)d_{H_0}(a_i)\geq (1-1/\sqrt{\a})d_{H_0}(a_i)$.
\end{proof}

\begin{indexedclaim}\label{clm:deg_in_B}
There exist a spanning subgraph $H'$ and a $\lceil\a \Delta\rceil$-colouring $\chi'$ of $A$ such that for every $a\in A$
\[
d_{H'}(a)\geq\left(1-\frac{1}{\sqrt{\a}}\right)d_{H_0}(a)\;,
\]
and for every $b\in B$
\[
d_{H'}(b)\geq\left(1-\frac{4}{\sqrt{\a}}\right) d_{H_0}(b)\;,
\]
and $N_{H'}(b)$ is rainbow in $\chi'$.
\end{indexedclaim}
\begin{proof}

Note that the subgraph $H_s$ and colouring $\chi_s$ we have constructed are random objects, so it suffices to show that they satisfy the required properties with positive probability (when $\Delta$ is large enough).
Note that two of the properties are guaranteed by the construction of $H_s$ and $\chi_s$ (partly using Claim~\ref{clm:deg_in_A}). It only remains to check the degree condition from $B$.

Let $b\in B$.
Observe that the number of coloured neighbours of $b$ under the colouring $\chi_s$ is at least the number of coloured neighbours of $b$ under $\chi_{s-1}$ (and so on), since in the deterministic phase an edge $ab$ can only be deleted in a step when $a$ is coloured.
Thus we can show that $d_{H_s}(b)$ is large by showing that the degree of $b$ in $H_1$ to the set of vertices coloured by $\chi_1$ is large.

For a given $a\in N_{H_0}(b)$, let $E_1$ be the event that there exists $a'\in N_{H_0}(b)\setminus\{a\}$ such that $\chi_0(a')=\chi_0(a)$ and let $E_2$ be the event that $a$ becomes uncoloured (as governed by the condition in~\eqref{cond:uncolour}).
Let $Y_b$ be the random variable that counts the number of vertices $a\in N_{H_0}(b)$ for which $E_1$ holds. Let $Z_b$ be the random variable that counts the number of vertices $a\in N_{H_0}(b)$ for which $E_2$ holds but $E_1$ does not. Notice that these random variables count disjoint sets of vertices. By the observation of the previous paragraph,
\[
d_{H_s}(b)\geq d_{H_0}(b)-Y_b-Z_b\;.
\]

We estimate $Z_b$ by studying another random variable.
We say that the colour $c$ is \emph{dangerous for $a$} if
\[
|\{b'\in N_{H_0}(a)\setminus\{b\}:\;\exists a'\in N_{H_0}(b')\setminus\{a\},\chi_0(a')=c\}|\geq\frac{d_{H_0}(a)}{\sqrt{\a}}-1\;.
\]
For a given $a\in N_{H_0}(b)$, let $E_3$ be the event that $a$ receives a dangerous colour.
Let $Z'_b$ be the random variable that counts the number of vertices $a\in N_{H_0}(b)$ for which $E_3$ holds but $E_1$ does not.

The following observation is important: if $a$ is counted by $Z_b$ it means that $a$ becomes uncoloured and $\chi_0(a)$ is a unique colour within $N_{H_0}(b)$. Then $a$ must have been assigned a dangerous colour since for every vertex $a'\in N_{H_0}(b)\setminus\{a\}$, $\chi_0(a')\neq\chi_0(a)$, and thus $a'$ does not change the number of $b'\in N_{H_0}(a)\setminus\{b\}$ that have colour $\chi_0(a)$ in $N_{H_0}(b')\setminus\{a\}$. Hence $Z_b\leq Z_b'$ and it is enough to verify that not too many vertices receive dangerous colours.

We are going to show that $X_b= Y_b+Z'_b$ is concentrated given any fixed colouring in $A\setminus N_{H_0}(b)$. This, together with an upper bound on the conditional expectation of $X_b$, suffices to establish an upper bound on $X_b$ that holds unconditionally. During the rest of the proof, we will assume that all the random variables are conditioned to the colouring in $A\setminus N_{H_0}(b)$.

First we deal with the expected value of $Y_b$. Consider $a\in N_{H_0}(b)$. Observe that at most $d_{H_0}(b)-1\leq\Delta$ colours appear in $N_{H_0}(b)\setminus\{a\}$ under the random colouring $\chi_0$. Then the probability that $a$ does not have a unique colour in $N_{H_0}(b)$ is at most
$(d_{H_0}(b)-1)/\lceil\a\Delta\rceil\leq 1/\a$, and so $\E(Y_b)\leq d_{H_0}(b)/\a$.

Second we compute the expected value of $Z'_b$. Since the maximum degree of $H_0$ is $\Delta$ and a colour is considered dangerous if at least $d_{H_0}(a)/\sqrt{\a}-1$ many vertices $b'\in N_{H_0}(a)\setminus\{b\}$ already have it in $N_{H_0}(b')\setminus \{a\}$, there are at most  $d_{H_0}(a)\Delta/(\Delta/\sqrt{\a}-1)\leq 2\sqrt{\a}\Delta$ dangerous colours for $a$. Thus $a$ receives a dangerous colour with probability at most $2\sqrt{\a}\Delta/\lceil\a\Delta\rceil\leq2/\sqrt{\a}$. So $\E(Z_b')\leq 2d_{H_0}(b)/\sqrt{\a}$.

Then
\[
\E(X_b)=\E(Y_b)+\E(Z_b')\leq\left(\frac{1}{\a}+\frac{2}{\sqrt{\a}}\right)d_{H_0}(b)\leq\frac{3d_{H_0}(b)}{\sqrt{\a}}\;.
\]

We can now apply the Simple Concentration Bound to show that $X_b$ is concentrated with polynomially small probability. Note that changing the colour of $a\in N_{H_0}(b)$ can change by at most two the value of $X_b$:
\begin{itemize}
\item[--]  it can change by at most two the number of vertices that are unique in their colour class  (including $a$ itself), and
\item[--]  it can change by at most one the number of vertices that receive a dangerous colour and do not satisfy $E_1$, since the colour classes are prescribed by the colouring given to $A\setminus N_{H_0}(b)$.
\end{itemize}
Moreover, $X_b$ conditioned on the colouring of $A\setminus N_{H_0}(b)$ is determined by at most $d_{H_0}(b)$ many different trials. By the Simple Concentration Bound with the choices $c=2$ and $n=d_{H_0}(b)$, we have that $X_b$ conditioned to any colouring in $A\setminus N_{H_0}(b)$ is unlikely to be large:
\[
\Pr\left(X_b\geq\frac{4d_{H_0}(b)}{\sqrt{\a}}\right)\leq \Pr\left(X_b-\E(X_b)\geq\frac{d_{H_0}(b)}{\sqrt{\a}}\right)\leq 2\exp\left(-\frac{d_{H_0}^2(b)}{8\a\cdot d_{H_0}(b)}\right) =e^{-\Omega(d_{H_0}(b))} = o(\Delta^{-6})\;.
\]
In the last equality we used that $d_{H_0}(b)=\Omega(\log^2{\Delta})$. Thus the previous inequality also holds for the unconditioned random variable $X_b$.

Observe that $X_b$ depends on the vertices at distance at most $3$ from $b$; the fact that $a\in N_{H_0}(b)$ retains its colour depends only on the colours assigned to vertices at distance $2$ from $a$. Thus every event corresponding to $X_b$ is mutually independent from the set of events corresponding to $X_{b'}$ with $b'$ at distance more than $6$ from $b$, the Lov\'asz Local Lemma yields that with positive probability $X_b\leq 4d_{H_0}(b)/\sqrt{\a}$ for every $b\in B$. This completes the proof of the claim.
\end{proof}

In the second round, we can apply the same argument to colour the vertices of $B$ using the subgraph $H'$. By Claim~\ref{clm:deg_in_B} and recalling that $\alpha > 16$, this graph has minimum degree at least $(1-4/\sqrt{\a})\delta(H_0)=\Omega(\log^2{\Delta})$ and maximum degree at most $\Delta$. So we can apply the same procedure (and claims) to colour $B$ with a new set of $\lceil\a\Delta\rceil$ colours. Combined with the colouring $\chi'$ of $A$, in this way we obtain a subgraph $H\subseteq H'$ and a $(2\lceil\a\Delta\rceil)$-colouring $\chi$ of $V$ such that
\begin{itemize}
\item for every $v\in V$
\[
d_H(v)\geq\left(1-\frac{4}{\sqrt{\a}}\right)^2 d_{H_0}(v)\geq\left(1-\frac{4}{\sqrt{\a}}\right)^2\frac{d_G(v)}{2}\;,\text{ and}
\]
\item $\chi$ is a $1$-frugal proper colouring of $H$.
\end{itemize}
This proves the theorem with the choice $\beta=\frac{1}{2}\left(1-4/\sqrt{\a}\right)^2$.
\end{proof}

\begin{corollary}\label{cor:decomp}
 Let $G$ be a graph with maximum degree $\Delta$ and minimum degree $\delta\geq\log^2{\Delta}$ with $\Delta$ sufficiently large. For every $\a > 16$, there exist $\beta=\beta(\a) > 0$ and $\ell\leq\lceil2\sqrt{2\lceil\a\Delta\rceil}\rceil$ many $C_4$-free disjoint spanning subgraphs $G_1,\dots,G_{\ell}$ such that for all $v\in V$
\[
\sum_{i=1}^{\ell} d_{G_i}(v)\geq\beta d_G(v)\;.
\]
\end{corollary}
\begin{proof}
We use the subgraph $H$ and the colouring $\chi$ guaranteed by Theorem~\ref{thm:v4} to find many $C_4$-free spanning subgraphs. By~\eqref{eq:clique_decomp}, for any sufficiently large $t$ there exists a decomposition of $K_t$ into $C_4$-free subgraphs $\G_1,\dots,\G_{\lceil2\sqrt{t}\rceil}$. Consider $t=2\lceil\a\Delta\rceil$ and for any $i\in[\lceil2\sqrt{t}\rceil]$ construct $G_i$ as follows:
\begin{itemize}
\item $V(G_i)=V(G)$ and
\item $uv\in E(G_i)$ if and only if $uv\in E(H)$ and $\chi(u)\chi(v)\in E(\G_i)$.
\end{itemize}

These subgraphs $G_i$ are disjoint and, since $H$ contains no monochromatic edge, each edge of $H$ appears in exactly one subgraph $G_i$. So the minimum degree condition for $H$ implies the  minimum degree sum condition demanded here. Moreover, each $G_i$ is $C_4$-free: by $\chi$ being $1$-frugal and proper, all $4$-cycles in $H$ are rainbow; and if $G_i$ contains such a $4$-cycle $C$, then the colours $\chi(C)$ form a $4$-cycle in $\G_i$.
\end{proof}

Besides the above, we need the following bound on arboricity by degeneracy (which follows, for instance, from the folkloric Proposition~3.1 of~\cite{NeOs12} combined with an old result of Nash-Williams~\cite{Nas64}).

\begin{lemma}\label{lem:arb}
Let $G = (V,E)$ be a graph with an ordering $(v_1,\dots,v_n)$ of $V$ which satisfies that $|N(v_i)\cap\{v_{i+1},\dots,v_{n}\}|\leq k$ for all $i\in [n]$.
Then $E$ can be partitioned into $k$ parts such that no part contains a cycle of $G$.
\end{lemma}

\begin{proof}[Proof of Theorem~\ref{thm:main}]
Let $G=(V,E)$ be a graph with maximum degree $\Delta$ and fix $\a> 16$.
We perform the following procedure.
\begin{enumerate}
	\item Let $\tilde{G}^0=G$ and $G' = (V,\emptyset)$.
	\item Start with $i=0$ and repeat the following until $i=\tau$, where $\tau$ is the smallest such that $\Delta(\tilde{G}^{\tau})\le\log^2{\Delta}$:
	\begin{enumerate}
		\item obtain $G^i$ from $\tilde{G}^i$ by successively removing all vertices of degree less than $\log^2{\Delta}$, and adding all of their incident edges to $G'$;
		\item apply Corollary~\ref{cor:decomp} to $G^i$ to obtain the disjoint $C_4$-free subgraphs $G^i_1, G^i_2,\dots, G^i_{\lceil2\sqrt{2\lceil\a\Delta(G^i)\rceil}\rceil}$;
		\item set $\tilde{G}^{i+1}=(V(G^i),E(G^i)\setminus\bigcup_j E(G^i_j))$ and then increment $i$.
\end{enumerate}
\item Add all edges of $\tilde{G}^{\tau}$ to $G'$.
\end{enumerate}

We can always apply Corollary~\ref{cor:decomp} at each iteration since in Step 2(a) we forced the minimum degree of $G^i$ to be at least $\log^2{\Delta}\ge\log^2{\Delta(G^i)}$.

Let us see that the maximum degree $\Delta(G^{i+1})$ is significantly smaller than $\Delta(G^i)$. By Corollary~\ref{cor:decomp}, the removal of $C_4$-free subgraphs at iteration $i$ removes at least $\beta d_{G^i}(v)$ edges incident to $v\in V$. Thus
\begin{align}\label{eq:2}
\Delta(G^{i+1})\leq\Delta(\tilde{G}^{i+1})\leq (1-\beta)\Delta(G^i)\leq (1-\beta)^i\Delta\;.
\end{align}
This implies that the procedure is guaranteed to stop after $\tau=O(\log\Delta)$ iterations.

Step 2(b) of each iteration generates a number of disjoint spanning $C_4$-free subgraphs, each of which we give a new colour. During the $i$th iteration we produce $\lceil2\sqrt{2\lceil\a\Delta(G^i)\rceil}\rceil < 2\sqrt{2\a\Delta(G^i)}+4$ such subgraphs, so by~\eqref{eq:2} and the bound on the number $\tau$ of iterations we produce at most
\begin{align}\label{eq:1}
O(\log \Delta)+
2\sqrt{2\a\Delta}+2\sqrt{2\a(1-\beta)\Delta}+2\sqrt{2\a(1-\beta)^2\Delta}+\dots = \frac{2\sqrt{ 2\a}}{1-\sqrt{1-\beta}}\cdot\sqrt{\Delta}+O(\log \Delta)
\end{align}
$C_4$-free subgraphs throughout all iterations.

It only remains to upper bound the number of colours needed in the remainder graph $G'$. By construction, $G'$ admits a degeneracy ordering satisfying the hypothesis of Lemma~\ref{lem:arb} for $k=\log^2{\Delta}$. Thus we can partition its edges into at most $\log^2{\Delta}$ acyclic (and thus $C_4$-free) subgraphs.
By~\eqref{eq:1} we obtain a partition of $E$ into $O(\sqrt{\Delta})$ $C_4$-free subgraphs in total.
This completes the proof of the theorem.
\end{proof}

\section*{Acknowledgement}
We thank the referees for their helpful comments and suggestions.

\bibliography{edgeC4free}

\providecommand{\bysame}{\leavevmode\hbox to3em{\hrulefill}\thinspace}
\providecommand{\MR}{\relax\ifhmode\unskip\space\fi MR }
\providecommand{\MRhref}[2]{%
  \href{http://www.ams.org/mathscinet-getitem?mr=#1}{#2}
}
\providecommand{\href}[2]{#2}
\begin{thebibliography}{10}

\bibitem{BEL76}
S.~A. Burr, P.~Erd{\H{o}}s, and L.~Lov\'asz, \emph{On graphs of {R}amsey type},
  Ars Combinatoria \textbf{1} (1976), no.~1, 167--190.

\bibitem{ChGr75}
F.~R.~K. Chung and R.~L. Graham, \emph{On multicolor {R}amsey numbers for
  complete bipartite graphs}, J. Combinatorial Theory Ser.~B \textbf{18}
  (1975), 164--169.

\bibitem{Erd38}
P.~Erd\H{o}s, \emph{On sequences of integers no one of which divides the
  product of two others and on some related problems},
  Inst.~Math.~Mech.~Univ.~Tomsk \textbf{2} (1938), 74--82.

\bibitem{FuSi13}
Z.~F{\"u}redi and M.~Simonovits, \emph{The history of degenerate (bipartite
  extremal graph problems)}, Erd\H{o}s Centennial, Bolyai Soc.~Math.~Stud.,
  vol.~25, J\'anos Bolyai Math.~Soc., Budapest, 2013, pp.~169--264.

\bibitem{Irv74}
R.~W. Irving, \emph{Generalised {R}amsey numbers for small graphs}, Discrete
  Math. \textbf{9} (1974), 251--264.

\bibitem{JMW13}
T.~Jiang, K.~G. Milans, and D.~B. West, \emph{Degree {R}amsey numbers for
  cycles and blowups of trees}, European J. Combin. \textbf{34} (2013), no.~2,
  414--423.

\bibitem{KMW12}
W.~B. Kinnersley, K.~G. Milans, and D.~B. West, \emph{Degree {R}amsey numbers
  of graphs}, Combin.~Probab.~Comput. \textbf{21} (2012), no.~1-2, 229--253.

\bibitem{LL09}
Y.~Li and K.~W. Lih, \emph{Multi-color {R}amsey numbers of even cycles},
  European Journal of Combinatorics \textbf{30} (2009), no.~1, 114--118.

\bibitem{MoRe02}
M.~Molloy and B.~Reed, \emph{Graph colouring and the probabilistic method},
  Algorithms and Combinatorics, vol.~23, Springer-Verlag, Berlin, 2002.

\bibitem{Nas64}
C.~St.~J.~A. Nash-Williams, \emph{Decomposition of finite graphs into forests},
  J. London Math.~Soc. \textbf{39} (1964), 12.

\bibitem{NeOs12}
J.~Ne{\v{s}}et{\v{r}}il and P.~Ossona~de Mendez, \emph{Sparsity}, Algorithms
  and Combinatorics, vol.~28, Springer, Heidelberg, 2012, Graphs, structures,
  and algorithms.

\bibitem{PeRe14+}
G.~Perarnau and B.~Reed, \emph{Existence of spanning $\mathcal{F}$--free
  subgraphs of regular graphs with large minimum degree}, arXiv:1404.7764
  (2014).

\bibitem{Pik12}
O.~Pikhurko, \emph{A note on the {T}ur\'an function of even cycles},
  Proc.~Amer.~Math.~Soc. \textbf{140} (2012), no.~11, 3687--3692.

\end{thebibliography}

\bibliographystyle{amsplain}

\end{document}